\documentclass[11pt, draft]{amsart}[2000]

\usepackage{amsmath}
\usepackage{amssymb}
\usepackage{amsthm}  
\usepackage[all]{xypic}
\usepackage{enumitem}
\usepackage{mathtools}
\usepackage{color}
\usepackage{bbm}
\usepackage{tikz-cd}
\usepackage{soul}
\usepackage[all]{xy}

\theoremstyle{plain}

\newtheorem{Thm}{Theorem}[section]
\newtheorem{Cor}[Thm]{Corollary}

\newtheorem{Prop}[Thm]{Proposition}

\theoremstyle{definition}
\newtheorem{Def}[Thm]{Definition}

\newtheorem{Exl}[Thm]{Example}
\newtheorem{Rmk}[Thm]{Remark}
\newtheorem{Ques}[Thm]{Question}

\numberwithin{equation}{section}

\newcommand{\A}{{\mathcal{A}}}

\newcommand{\F}{{\mathcal{F}}}

\newcommand{\I}{{\mathcal{I}}}
\newcommand{\J}{{\mathcal{J}}}

\renewcommand{\L}{{\mathcal{L}}}

\renewcommand{\O}{{\mathcal{O}}}

\newcommand{\Q}{{\mathcal{Q}}}

\newcommand{\T}{{\mathcal{T}}}

\newcommand{\bF}{\mathbb{F}}
\newcommand{\bN}{\mathbb{N}}
\newcommand{\bT}{\mathbb{T}}
\newcommand{\bZ}{\mathbb{Z}}

\newcommand{\lip}{\langle}
\newcommand{\rip}{\rangle}

\newcommand{\Aut}{\operatorname{Aut}}

\newcommand{\Gr}{\operatorname{Gr}}

\newcommand{\id}{{\operatorname{id}}}

\DeclareMathOperator{\supp}{supp}
\DeclareMathOperator{\Arv}{Arv}
\DeclareMathOperator{\Diag}{Diag}

\DeclareMathOperator{\LL}{LL}

\newcommand{\bigslant}[2]{{\raisebox{.2em}{$#1$}\left/\raisebox{-.2em}{$#2$}\right.}}

\begin{document}
	
	\title[Toeplitz quotients and ratio limits]{Toeplitz quotient C*-algebras \\ and ratio limits for random walks}
	
	\author[A. Dor-On]{Adam Dor-On}
	\address{Department of Mathematics and Informatics \\ University of M\" unster \\ Germany.}
	\email{adoron.math@gmail.com}

	
	\thanks{The author was partially supported by NSF grant DMS-1900916 and by the European Union's Horizon 2020 Marie Sklodowska-Curie grant No 839412. }

	\begin{abstract}
	We study quotients of the Toeplitz C*-algebra of a random walk, similar to those studied by the author and Markiewicz for finite stochastic matrices. We introduce a new Cuntz-type quotient C*-algebra for random walks that have convergent ratios of transition probabilities. These C*-algebras give rise to new notions of ratio limit space and boundary for such random walks, which are computed by appealing to a companion paper by Woess. Our combined results are leveraged to identify a unique smallest symmetry-equivariant quotient C*-algebra for any symmetric random walk on a hyperbolic group, shedding light on a question of Viselter on C*-algebras of subproduct systems.
	\end{abstract}
	
%%%%%%%%%%%%%%%%%%%%%%%%%%%%%%%%
		
\maketitle

%%%%%%%%%%%%%%%%%%%%%%%%%%%%%%%%
\renewcommand{\O}{{\mathcal{O}}}
	
\section{Introduction}
	
It is an age-old tradition, since the work of Murray and von Neumann \cite{MvN36, MvN37}, to use operator algebras as means of producing new invariants in various theories in Mathematics. One instance where the theory of C*-algebras was useful in this regard is in the classification of Cantor minimal $\mathbb{Z}^d$ systems up to orbit equivalence through the use of K-theoretical invariants, leading to new notions of equivalence relations between the systems \cite{GPS95, GMPS10}.

Another instance of this is in graph theory and symbolic dynamics, where invariants of C*-algebras studied by Cuntz and Krieger coincide with invariants coming from subshifts of finite type \cite{CK80, Cu81}. After contributions and improvements by too many authors to list here, these works led to C*-algebraic interpretations of equivalence relations occurring naturally in symbolic dynamics \cite{MM14, CR17}, and provided a rich class of examples for classification of operator algebras \cite{ERRS+, DEG2020}.

A concrete way of constructing and studying C*-algebras of directed graphs is by realizing them as unique smallest $\mathbb{T}$-equivariant quotients of the Toeplitz C*-algebras of the graph. These Toeplitz C*-algebras are simply those generated by concatenation operators on the space of square-summable sequences indexed by all finite paths of the graph. Such concrete realizations, together with previous works on C*-algebras of subproduct systems \cite{Vis11, Vis12}, allowed us to reveal the structure of Toeplitz C*-algebras and tensor operator algebras of subproduct systems arising from stochastic matrices \cite{DOM14,DOM16}. 

In this paper, we introduce a new Cuntz-type C*-algebra $\O(G,\mu)$ for a random walk $P$ on a group $G$ induced by a finitely supported measure $\mu$, which is a quotient of the Toeplitz algebra $\T(G,\mu)$ of the stochastic matrix $P$. The computation of $\O(G,\mu)$ in this paper gave rise to new notions of ratio-limit space and boundary for random walks, prompting the study in the companion paper by Woess \cite{Woe+}. When the random walk is finite, our Cuntz C*-algebras coincide with the ones computed in \cite[Theorem 2.1]{DOM16}, but new subtleties emerge for random walks on infinite groups.

For a stochastic matrix $P$ on a group $G$, we denote by $P^{(n)}_{x,y}:=(P^n)_{x,y}$ the $n$-step transition probability from $x$ to $y$, for $x,y\in G$.

\begin{Def}
Let $P$ be an irreducible stochastic matrix over $G$. We say that $P$ has the \emph{strong ratio limit property} (SRLP) if for all $x,y,z \in \Omega$ we have that the limit $\underset{m \rightarrow \infty}{\lim} \frac{P^{(m)}_{x,y}}{P^{(m)}_{z,y}} $ exists.
\end{Def}

SRLP was first established for integer lattices in works of Chung and Erd\"{o}s \cite{CE51} and of Kesten \cite{Kes63}, and was later shown to hold for random walks over abelian groups \cite{Sto66}, random walks on nilpotent groups \cite{Mar66}, and symmetric random walks on amenable groups \cite{Ave73}. These days, establishing SRLP often relies on local limit theorems. More precisely, typical local limit theorems for $P$ determine the asymptotic behavior of $P_{x,y}^{(n)}$ in the sense that
$$
P^{(n)}_{x,y} \underset{n \rightarrow \infty}{\sim} C \cdot \beta(x,y) \cdot \rho^n \cdot n^{-\alpha},
$$
for $C, \beta(x,y), \alpha > 0$, where the ratio between the LHS and RHS goes to $1$ as $n\rightarrow \infty$ (see \cite{Rev03, BW05} for other kinds of local limit theorems). If we have a local limit theorem as above, we get SRLP where $\underset{m \rightarrow \infty}{\lim} \frac{P^{(m)}_{x,y}}{P^{(m)}_{z,y}} = \frac{\beta(x,y)}{\beta(z,y)}$.

Local limit theorems have been established for certain random walks on free products \cite{Woe86, Car88}, random walks on free groups and trees \cite{Lal93}, symmetric random walks on co-compact Fuchsian groups \cite{GL13} and symmetric random walks on non-elementary hyperbolic groups \cite{Gou14}. For more on the history of local limit theorems we refer the reader to \cite[Chapter III]{Woe00}, as well as the companion paper by Woess \cite{Woe+}. In general, it is unknown whether or not aperiodic random walks automatically satisfy SRLP.

Assuming SRLP, a ratio-limit space and boundary arise from the computation of $\O(G,\mu)$, leading to the following definitions in the theory of random walks. The \emph{ratio-limit kernel} $H : G \times G \rightarrow (0,\infty)$ is given by
$$
H(x,y) = \underset{m \rightarrow \infty}{\lim} \frac{P^{(m)}_{x,y}}{P^{(m)}_{e,y}},
$$
which turns out to be bounded in $y \in G$ for every fixed $x\in G$. We let $R_{\mu}$ be the largest subgroup of $G$ on which the functions $y \mapsto H(x,y)$ are constant for all $x\in G$. Then, we define the \emph{ratio-limit space} $\mathrm{R}(G,\mu)$ to be the smallest compactification of $G / R_{\mu}$ to which the functions $y \mapsto H(x,y)$ extend continuously for all $x\in G$. The \emph{ratio limit boundary} is given by 
$$
\partial_{\mathrm{R}} G = \mathrm{R}(G,\mu) \setminus [G / R_{\mu}].
$$

Let $\mathbb{T}$ be the unit circle, and denote by $\mathbb{K}(\ell^2(G))$ the compact operators on the Hilbert space $\ell^2(G)$. The following establishes the connection between $\O(G,\mu)$ and the ratio limit space $\mathrm{R}(G,\mu)$ in this work.

\begin{Thm}
Let $P$ be a random walk on a group $G$ induced by a finitely supported measure $\mu$, and assume that $P$ has SRLP. Then
$$
\O(G,\mu) \cong C(\mathrm{R}(G,\mu) \times \mathbb{T})\otimes \mathbb{K}(\ell^2(G)).
$$
\end{Thm}

This result prompted the computation of the ratio-limit boundary for several classes of examples in the companion paper by Woess \cite{Woe+}. This includes isotropic random walks on trees \cite[Theorem 3.3]{Woe+}, random walks on free groups \cite[Theorem 3.12]{Woe+}, and symmetric random walks on non-elementary hyperbolic groups \cite[Theorem 4.5]{Woe+}.

As a consequence, we are able to shed light on a questions of Viselter on C*-algebras associated with subproduct systems. Subproduct systems were introduced by Shalit and Solel in \cite{SS09} for the purpose of studying quantum Markov semigroups (see also \cite{MS02, BM10}), and for unifying the study of certain operator algebras of nc holomorphic functions (see for instance \cite{DRS11, DRS15, SSS18}). In work of Viselter \cite{Vis12}, Cuntz-Pimsner C*-algebras of a subproduct system were defined in a way that generalized essentially all previous examples.

In \cite[Section 6, Question 1]{Vis12} Viselter asked if his C*-algebras have a universal property in the spirit of a gauge-invariant uniqueness theorem. Gauge-invariant uniqueness theorems have a plethora of applications in the structure and representation theory of operator algebras, and have been extended significantly to various scenarios \cite{Pim97, Kat04, RSY04, Kat07, CLSV11, DK20, DKKLL+}. Hence, it is natural to ask for such uniqueness theorems in the context of subproduct systems. However, already in \cite[Example 2.3]{Vis12} it was shown that a \emph{unique} smallest $\mathbb{T}$-gauge equivariant quotient C*-algebra may fail to exist in general.

In a recent preprint of Arici and Kaad \cite{AK+} it is shown that subproduct systems arising from representations of $SU(2)$ give rise to a natural $SU(2)$-action on associated Toeplitz and Cuntz C*-algebras. These symmetries are leveraged to provide analogues of Gysin sequences that are used to compute the $K$-theory of these Toeplitz and Cuntz C*-algebras via Euler characteristic classes. Analogously to Viselter's question, in \cite[Section 8, Question 3]{AK+} it is asked whether Viselter's Cuntz-Pimsner C*-algebra is the unique smallest $SU(2)$-equivariant quotient for subproduct systems arising in \cite{AK+}. The key observation made by asking this question is that Viselter's Cuntz-Pimsner algebra in \cite[Example 2.3]{Vis12} turns out to be the unique smallest $SU(2)$-equivariant quotient of its respective Toeplitz C*-algebra.

Hence, Viselter's question can be interpreted as asking whether his Cuntz-Pimsner C*-algebras satisfy symmetry-uniqueness with respect to a natural class of symmetries on the Toeplitz algebra, at least in cases where a unique smallest equivariant quotient of Toeplitz algebra exists with respect to this class. 

We answer the above question in the negative, showing that Viselter's Cuntz-Pimsner C*-algebra has a proper quotient which is the unique smallest $G \times \mathbb{T}$ equivariant quotient of $\T(G,\mu)$. In fact, for symmetric aperiodic random walks on non-elementary hyperbolic groups, whose ratio-limit boundary is computed in the companion paper \cite{Woe+}, we get a unique smallest $G \times \mathbb{T}$-equivariant quotient of $\T(G,\mu)$ which is a proper quotient of both Viselter's C*-algebra and $\O(G,\mu)$.

\begin{Thm}
Let $P$ be a symmetric aperiodic random walk on a non-elementary hyperbolic group induced by a finitely supported measure $\mu$, and let $\partial G$ be the Gromov boundary of $G$. Then, the C*-algebra $C(\partial G \times \mathbb{T})\otimes \mathbb{K}(\ell^2(G))$ is the unique smallest $G \times \mathbb{T}$-equivariant quotient of $\T(G,\mu)$.
\end{Thm}

This paper has five sections, including this introduction. In Section \ref{s:stoch} we provide some of the necessary preliminaries on stochastic matrices and random walks. In Section \ref{s:ratio-limit space} we introduce the ratio-limit space and boundary of a random walk with SRLP arising in this work, and provide some examples by appealing to the companion paper by Woess \cite{Woe+}. In Section \ref{S:C*-alg-for-Markov-chains} we define Toeplitz and Cuntz algebras for random walks, and compute the latter under the assumption of SRLP. Finally, in Section \ref{s:symmetry-uniqueness} we find conditions on the ratio limit boundary to ensure uniqueness of smallest equivariant quotients, and explain how our setting transfers to the context of subproduct systems where we discuss consequences on Viselter's question.

\subsection*{Acknowledgments} The author is grateful to Wolfgang Woess for many helpful exchanges on the subject of random walks and their boundaries, for providing remarks on this paper, and for computing ratio limit boundaries in many classes of examples in the companion paper \cite{Woe+}. The author is also grateful to Christopher Linden and Alex Vernik for suggestions, discussions and remarks on draft versions of this paper.

%%%%%%%%%%%%%%%%%%%%%%%%%%%%%%%%
	
\section{Stochastic matrices and random walks.} \label{s:stoch}

In this subsection we discuss some of the needed theory on stochastic matrices and random walks. For more on the relevant theory we recommend the survey \cite{Woe94} and the books \cite{Woe00, Woe09}.

\begin{Def}
Let $\mathcal{X}$ be a countable set. A \emph{stochastic matrix} over $\mathcal{X}$ is a map $P : \mathcal{X} \times \mathcal{X} \rightarrow [0,1]$ such that $\sum_j P_{ij} = 1$. We let $\Gr(P)$ be the directed graph on $\mathcal{X}$ with directed edges
$E(P):= \{ \ (i,j) \ | \ P_{ij}>0 \ \}$. We say that $P$ is \emph{irreducible} if $\Gr(P) = (\mathcal{X},E(P))$ is a strongly connected directed graph.
\end{Def}

When $P$ is a stochastic matrix over $\mathcal{X}$, we denote by $P^n$ the $n$-th iterate of $P$, and by $P^{(n)}_{ij}$ the $ij$-th entry of $P^n$. We denote $P^0 := I$ the identity matrix. We say that $P$ is \emph{symmetric} when it is equal to its transpose. We also say that $P$ is \emph{aperiodic} if the greatest common divisor of lengths of all cycles in $\Gr(P)$ is $1$. We will assume henceforth that $\mathcal{X}$ is countable.

\begin{Def}
Let $P$ be an irreducible stochastic matrix over $\mathcal{X}$. The \emph{spectral radius} of $P$ is given by
$$
\rho(P):= \limsup_{n\rightarrow \infty} \sqrt[n]{P^{(n)}_{ij}}
$$
and is independent of $i,j\in \mathcal{X}$.
\end{Def}

We denote by $\rho:=\rho(P)$ when the context is clear. We will say that a non-negative function $h : \mathcal{X} \rightarrow [0,\infty)$ is \emph{$\rho$-harmonic} at $i\in \mathcal{X}$ if $(Ph)(i):= \sum_{j\in \mathcal{X}} P_{ij}h(j) = \rho \cdot h(i)$. The Green kernel of $P$ is given for $i,j \in \mathcal{X}$ by
$$
G(i,j|z) = \sum_{n=0}^{\infty}P^{(n)}_{ij}z^n,
$$
with radius of convergence $\rho^{-1}$. We denote also $F(i,j|z):= \frac{G(i,j|z)}{G(j,j|z)}$, so that by \cite[Lemma 3.66]{Woe09} we get that $\lim_{z \rightarrow \rho^{-1}} F(i,j|z)$ exists for every $i,j\in \mathcal{X}$. Let $o \in \mathcal{X}$ be some fixed element. We define the $\rho$-Martin kernel of $P$ to be
$$
K(i,j) := \lim_{z\rightarrow \rho^{-1}} \frac{G(i,j|z)}{G(o,j|z)} = \lim_{z\rightarrow \rho^{-1}} \frac{F(i,j|z)}{F(o,j|z)},
$$
which exists and is finite. For fixed $j\in \mathcal{X}$, the function $i \mapsto K(i,j)$ is then $\rho$-harmonic at all points, except when $i=j$, while for fixed $i\in \mathcal{X}$ the function $j \mapsto K(i,j)$ is bounded above and away from $0$.

Now let $\phi : \mathcal{X} \rightarrow \bN$ be some bijection. The $\rho$-Martin compactification $\Delta_{\rho}\mathcal{X}$ is the completion of $\mathcal{X}$ with respect to the metric
$$
d(j_1,j_2) = \sum_{i\in \mathcal{X}} \frac{|K(i,j_1) - K(i,j_2)| + |\delta_{ij_1} - \delta_{ij_2}|}{C_i \cdot 2^{\phi(i)}}.
$$
Then, $\Delta_{\rho} \mathcal{X}$ becomes the smallest compactification of $\mathcal{X}$ to which the functions $i \mapsto K(i,j)$ extend continuously for every fixed $j\in \mathcal{X}$, and contains $\mathcal{X}$ as an open subset (see for instance \cite[Theorem 7.13]{Woe00} for an equivalent construction). A sequence $\alpha_n \in \mathcal{X}$ converges to a $\alpha \in \Delta_{\rho}\mathcal{X}$ if either $\alpha \in \mathcal{X}$ and $\alpha_n$ is eventually equal to $\alpha$, or $\alpha_n$ is eventually outside any finite set and $\lim_n K(i,\alpha_n) = K(i,\alpha)$ for every $i \in \mathcal{X}$. The closed subspace $\partial_{\Delta, \rho} \mathcal{X} = \Delta_{\rho}\mathcal{X} \setminus \mathcal{X}$ is called the $\rho$-Martin boundary of $P$.

Our focus in this work will be on irreducible stochastic matrices that are random walks on groups, with finitely supported measures.

\begin{Def}
Let $G$ be a countable discrete group, and $\mu : G \rightarrow [0,1]$ a finitely supported probability measure such that $\supp(\mu)$ generates $G$ as a semigroup. The stochastic matrix $P$ on $G$ given by $P_{x,y} = \mu(x^{-1}y)$ is called the \emph{random walk} on $G$ induced by $\mu$.
\end{Def}

The iterates of $P$ are then given by $P_{x,y}^{(n)} = \mu^{* n}(x^{-1}y)$ where $\mu^{*n}$ is the $n$-th convolution power of $\mu$. Note also that $P$ is symmetric if and only if $\mu(g) = \mu(g^{-1})$ for every $g\in G$, and that $P$ is aperiodic if and only if there is some odd $n$ such that $\mu^{*n}(e) > 0$.

The main reason for choosing a finitely supported measure $\mu$ in the above definition, is to assure that the random walk $P$ has finite range, or alternatively, that the graph $\Gr(P)$ is locally finite. More precisely, for any fixed $z\in G$, there are finitely many $y\in G$ such that $P_{y,z} > 0$.

One of the defining features of random walks is that they have symmetries coming from a group. That is, for every $g\in G$ we have that $P_{gx,gy}^{(n)} = P_{x,y}^{(n)}$. This gives rise to $G$-invariance of the Green kernel in the sense that for every $g\in G$ and $x,y\in G$ and $0< z < \rho^{-1}$ we have $G(gx,gy|z) = G(x,y|z)$, and for $0 < z \leq \rho^{-1}$ we have $F(gx,gy|z) = F(x,y|z)$. But then, since $ K(x,gy) = K(g^{-1}x,y) / K(g^{-1},y)$, we see that the left multiplication map $\alpha_g: x \mapsto gx$ is continuous with respect to the metric $d$. Hence, $\alpha_g$ extend to a homeomorphism (still denoted) $\alpha_g$ on $\Delta_{\rho}G$. Furthermore, $\alpha_g$ clearly maps $G$ onto itself, and so must map $\partial_{\Delta, \rho}G$ onto itself as well. Thus, when $P$ is a random walk, we see that the compacta $\Delta_{\rho}G$ and $\partial_{\Delta, \rho}G$ both carry $G$-actions by homeomorphisms induced by left multiplication on $G$.

\section{Ratio limit space and boundary.} \label{s:ratio-limit space}

Recall an essential assumption for random walks and their operator algebras, that will be used throughout this paper.

\begin{Def}
Let $P$ be a random walk on $G$ induced by a finitely supported measure $\mu$. We say that $P$ has the \emph{strong ratio limit property} (SRLP) if for all $x,y,z \in G$ we have that $\lim_{m \rightarrow \infty} \frac{P^{(m)}_{x,y}}{P^{(m)}_{z,y}} $ exists.
\end{Def}

Note that if these limits exist and are all non-zero, this implies that $P$ is aperiodic, so that for any $x,y \in G$ we have $n_0$ such that for all $n\geq n_0$ one must have $P^{(n)}_{x,y} > 0$.  

Suppose now that $P$ is an aperiodic random walk on a group $G$ induced by a finitely supported measure $\mu$. By \cite[Satz 1]{Ger78} (see also \cite[Proposition 7.1]{Han99}) we know that $\lim_{m \rightarrow \infty} \frac{\mu^{*(m+1)}(x)}{\mu^{*m}(x)} = \rho$ for every $x \in G$ where $\rho = \rho(P)$ is the spectral radius. Thus, the limiting behavior of the sequences $\Big \{ \frac{\mu^{*m}(x^{-1}y)}{\mu^{*m}(y)}\Big \}$ is comparable with some other better-behaved sequences. Indeed, when $\mu^{*n}(y) > 0$ and $\mu^{*n'}(x^{-1}y) > 0$ we get that
$$
\frac{\mu^{*m}(x^{-1}y)}{\mu^{*m}(y)} \underset{m\rightarrow \infty}{\sim} \frac{\rho(P)^n \mu^{*m}(x^{-1}y)}{\mu^{*(m+n)}(y)} \underset{m\rightarrow \infty}{\sim} \frac{\mu^{*(m+n')}(x^{-1}y)}{\rho(P)^{n'} \cdot \mu^{*m}(y)}
$$
The advantage of doing this, is that we can assure that eventually $\frac{\mu^{*m}(x^{-1}y)}{\mu^{*m}(y)}$ is bounded above and away from $0$ for every fixed $x \in G$. Indeed, for smallest $n,n'\in \bN$ such that $\mu^{*n}(x), \mu^{*n'}(x^{-1}) >0$ we get that 
$$
\frac{\rho(P)^n \mu^{*m}(x^{-1}y)}{\mu^{*(m+n)}(y)} \leq C_x , \ \ \text{and} \ \ \frac{\mu^{*(m+n')}(x^{-1}y)}{\rho(P)^{n'} \cdot \mu^{*m}(y)} \geq c_x.
$$
where
$$
C_x = \frac{\rho(P)^n}{\mu^{*n}(x)} \ \text{and} \ c_x = \frac{\mu^{*n'}(x^{-1})}{\rho(P)^{n'}}.
$$
Hence, for fixed $x\in G$ and sufficiently large $m$ we get $c_x \leq \frac{\mu^{*m}(x^{-1}y)}{\mu^{*m}(y)} \leq C_x$ for every $y \in G$. 

Suppose now that $P$ has SRLP. From the $G$-symmetry of the random walk, this is equivalent to the existence of the limits $\lim_m \frac{\mu^{*m}(x)}{\mu^{*m}(e)}$ for each $x\in G$. We may define the \emph{ratio limit kernel} $H : G \times G \rightarrow (0,\infty)$ given by 
$$
H(x,y) = \lim_m \frac{\mu^{*m}(x^{-1}y)}{\mu^{*m}(y)}.
$$
Then, by the above $x \mapsto H(x,y)$ is $\rho$-harmonic for every fixed $y\in G$ and $y \mapsto H(x,y)$ is bounded and bounded away from $0$ for every fixed $x\in G$. For each $x\in X$, we denote by $H(x,\cdot)$ the \emph{ratio-limit function} $y \mapsto H(x,y)$.

\begin{Prop} \label{p:rlr}
Let $P$ be a random walk on a group $G$ induced by a finitely supported measure $\mu$. Suppose that $P$ has SRLP. Then the set
$$
R_{\mu}:= \{ \ y \in G \ | \ H(x,y) = H(x,e) \ \forall x\in G \ \}
$$
is a subgroup of $G$.
\end{Prop}

\begin{proof}
For $y,z\in R_{\mu}$ and $x\in G$ we have
$$
H(x,yz) = \lim_m \frac{\mu^{*m}(x^{-1}yz)}{\mu^{*m}(yz)} = \lim_m \frac{\mu^{*m}((y^{-1}x)^{-1}z)}{\mu^{*m}(z)} \cdot \frac{\mu^{*m}(z)}{\mu^{*m}((y^{-1})^{-1}z)} =
$$
$$
= H(y^{-1}x,z) H(y^{-1},z)^{-1} = H(y^{-1}x,e)H(y^{-1},e)^{-1} = 
$$
$$
\lim_m \frac{\mu^{*m}(x^{-1}y)}{\mu^{*m}(e)} \cdot \frac{\mu^{*m}(e)}{\mu^{*m}(y)} = H(x,y) = H(x,e),
$$
and we also have
$$
H(x,y^{-1}) = \lim_m \frac{\mu^{*m}((yx)^{-1})}{\mu^{*m}(y^{-1})} = \lim_m \frac{\mu^{*m}((yx)^{-1})}{\mu^{*m}(e)} \frac{\mu^{*m}(e)}{\mu^{*m}(y^{-1})} =
$$
$$
H(yx,e)H(y,e)^{-1} = H(yx,y)H(y,y)^{-1} = 
$$
$$
\lim_m \frac{\mu^{*m}(x^{-1})}{\mu^{*m}(y)}\frac{\mu^{*m}(y)}{\mu^{*m}(e)} = H(x,e).
$$
\end{proof}

We call $R_{\mu}$ the \emph{ratio-limit radical}, as it is the largest subgroup of $G$ on which the ratio-limit functions $\{H(x,\cdot)\}_{x\in G}$ are constant. We let $G / R_{\mu}$ be the left cosets of $G$ by $R_{\mu}$. Note that the ratio-limit functions are well-defined on, and separate points in $G / R_{\mu}$.

\begin{Rmk}
When $P$ as above is also symmetric, there is a subgroup $A_{\mu} \leq G$ defined in \cite{ER19} given by
$$
A_{\mu} = \Big\{ \ y \in G \ \big| \ \lim_m \frac{\mu^{*m}(y)}{\mu^{*m}(e)} = 1 \ \Big\},
$$
which is amenable by \cite[Theorem 4.2]{ER19}. Together with SRLP, from the definition of $R_{\mu}$ we see that for any $y\in R_{\mu}$ we have $H(y,y) = H(y,e)$. Hence, by symmetry of $P$ we dedudce that $y \in A_{\mu}$, so that $R_{\mu}$ is a subgroup of $A_{\mu}$. Hence, when $P$ is symmetric we get that $R_{\mu}$ is amenable.
\end{Rmk}

\begin{Def}
Let $P$ be a random walk on a group $G$ induced by finitely supported $\mu$. Suppose that $P$ satisfies SRLP. The (reduced) \emph{ratio-limit space} $\mathrm{R}(G,\mu)$ is the smallest compactification of $G / R_{\mu}$ which makes the ratio limit functions $\{H(x,\cdot)\}_{x\in G}$ extend continuously to $\mathrm{R}(G,\mu)$. More precisely, if $\phi : G \rightarrow \bN$ is some bijection, then $\mathrm{R}(G,\mu)$ is the completion of $G / R_{\mu}$ with respect to the bounded metric 
$$
d(yR_{\mu},zR_{\mu}) = \sum_{x\in G} \frac{|H(x,y) - H(x,z)|}{C_x \cdot 2^{\phi(x)}}.
$$
The subspace $\partial_{\mathrm{R} }G = \mathrm{R}(G,\mu) \setminus (G / R_{\mu})$ is called the (reduced) \emph{ratio-limit boundary} of the random walk.
\end{Def}

It follows from general topology (see \cite[Theorem 3.5.8]{Eng89}) that $G / R_{\mu}$ is open in $\mathrm{R}(G,\mu)$, so that $\partial_{\mathrm{R} }G$ is a closed subspace. The topology on $\mathrm{R}(G,\mu)$ is determined by specifying that a sequence $y_n \in G$ converges to a point $y \in \mathrm{R}(G,\mu)$ if either $y \in G$ and $y_n\sim y$ for eventually every $n$, or that $y \in \partial_{\mathrm{R}} G$ and $\lim_n H(x,y_n) = H(x,y)$ for every $x\in G$. Furthermore, since $H(x,gy) = H(g^{-1}x,y) / H(g^{-1},y)$, again we get that left multiplication $\beta_g : xR_{\mu} \mapsto gxR_{\mu}$ on $G / R_{\mu}$ is continuous with respect to $d$, and extends to a homeomorphism (still denoted) $\beta_g$ on $\mathrm{R}(G,\mu)$. Hence, as before, we get that the compacta $\mathrm{R}(G,\mu)$ and $\partial_{\mathrm{R}}G$ carry $G$-actions by homeomorphisms induced by left multiplication on $G$ (for the latter when it is non-empty).

When $G$ is an amenable group, and $P$ is a symmetric aperiodic random walk on $G$ induced by a finitely supported $\mu$, by Avez' theorem \cite{Ave73} (see also \cite[Corollary 3.3]{ER19}) we get for any $x\in G$ that $\lim \frac{\mu^{*m}(x)}{\mu^{*m}(e)} = 1$. In this case $R_{\mu} = G$, so that $\mathrm{R}(G,\mu) = G / R_{\mu}$ is trivial, and the ratio limit boundary is empty. Together with this, the next example shows that the ratio limit boundary / space may fail to coincide with the $\rho$-Martin boundary in general.

\begin{Exl}[Random walks on lamplighter groups] \label{E:lamplighter}
Let $\LL(\bZ^d) = \bZ^d \times \big[\bigoplus_{x \in \bZ^d}\bZ_2 \big]$ where $\bigoplus_{x \in \bZ^d}\bZ_2$ are finitely supported functions on $\bZ^d$ with $d\geq 3$. Then, $\LL(\bZ^d)$ has group multiplication given by $(x,w) \cdot (y,u) = (x+y, w+ T_x(u))$ where $T_x(u)$ is given by $T_x(u)(z) = u(z-x)$. Let $P$ be an aperiodic symmetric random walk on $\LL(\bZ^d)$ induced by a finitely supported measure $\mu$. From \cite[Example 6.1]{Kai83} we know that $\LL(\bZ^d)$ is amenable, so that by Kesten's amenability criterion \cite{Kes59} we get that the spectral radius of $P$ is $\rho=1$. On the other hand, by \cite[Proposition 6.1]{Kai83} we also get that $\mu$ has a non-trivial Poisson boundary. Since the Poisson boundary is contained in the $1$-Martin boundary, we see that $P$ has non-trivial $\rho$-Martin boundary while having a trivial ratio limit space and empty ratio limit boundary.
\end{Exl}

\begin{Exl} \label{E:free}
Let $\bF_s$ be the free group on $s$ generators $a_1,...,a_s$, and let $d$ be the shortest path metric on the Cayley graph $\mathrm{T}$ of $\bF_s$ with respect to the symmetric generating set $S = \{a_1,...,a_s, a_1^{-1},...,a_s^{-1}\}$. Note that $\mathrm{T} = \mathrm{T}_{2s}$ is just the $2s$ regular tree. We take a (finitely supported) probability measure $\mu$ on $\bF_s$ with $\mu(e) > 0$, which is a function $\mu(w) = f(d(e,w))$ of the distance of $w\in \bF_s$ to the identity element $e \in \bF_s$ in $\mathrm{T}$. Then $\mu$ induces what is known as an \emph{isotropic} random walk on $\bF_s$. By the local limit theorem of Sawyer \cite{Saw78} (see also \cite[Theorem 19.30]{Woe00}), we have that 
$$
P^{(n)}_{x,y} \underset{n \rightarrow \infty}{\sim} C \cdot \beta(x,y) \cdot \rho^n \cdot n^{-3/2},
$$
where $\beta(x,y) = (1 + \frac{s-1}{s}d(x,y))(2s-1)^{-d(x,y)/2}$. Hence, for $x,y\in \bF_s$ we get a formula for the ratio-limit kernel,
$$
H(x,y) = \frac{1 + \frac{s-1}{s}d(x,y)}{1 + \frac{s-1}{s}d(e,y)}(2s-1)^{\frac{d(e,y)-d(x,y)}{2}}.
$$
In particular, we see that the ratio-limit functions separate points in $\bF_s$, so that $R_{\mu}$ is trivial. Hence, $\bF_s$ embeds into $\mathrm{R}(\bF_s,\mu)$, and by \cite[Theorem 3.3]{Woe+} the ratio limit boundary $\partial_{\mathrm{R}} \bF_s$ coincides with the space of ends $\partial \mathrm{T}_{2s}$ of the $2s$-regular tree $\mathrm{T}_{2s}$. Hence, we get that $\mathrm{R}(\bF_s,\mu) = \bF_s \cup \partial \mathrm{T}_{2s}$.
\end{Exl}

\begin{Exl}
Let $P$ be an aperiodic isotropic random walk on $\bF_{s_1}$ arising from $\mu_1$ and $Q$ be an apriodic symmetric random walk on $\mathbb{Z}^{s_2}$ arising from $\mu_2$ where $s_1 \geq 2$ and $s_2 \geq 1$. Take the Cartesian product $G = \bF_{s_1} \times \mathbb{Z}^{s_2}$, and let $\pi_1 : G \rightarrow \bF_{s_1}$ and $\pi_2 : G \rightarrow \mathbb{Z}^{s_2}$ be the coordinate projections. By \cite[Theorem 13.12]{Woe00} we have that $Q$ satisfies a local limit theorem of the form
$$
Q^{(n)}_{v_1,v_2} \underset{n \rightarrow \infty}{\sim} C \cdot n^{-s_2/ 2}.
$$
Next, define the \emph{Cartesian} random walk by setting $\mu = \frac{1}{2}[\mu_1 \circ \pi_1^{-1} + \mu_2 \circ \pi_2^{-1}]$ where $\mu_1$, where $\mu_1 \circ \pi_1^{-1}$ and $\mu_2 \circ \pi_2^{-1}$ are pushforward measures. Then, by \cite[Proposition 5.3]{Woe+} (see also \cite{CS87}), we get that the ratio limit kernel $H$ for $\mu$ is given by $H((w_1,v_1),(w_2,v_2)) = H_1(w_1,w_2)H_2(v_1,v_2)$, where $H_1$ and $H_2 = 1$ are the ratio limit kernels of $P$ and $Q$ respectively. By Example \ref{E:free} and the formula for $H_1$ there, so we get that $R_{\mu} = R_{\mu_2} = \mathbb{Z}^{s_2}$. Thus, $G / R_{\mu}$ coincides with $\bF_{s_1}$, and the ratio limit space $\mathrm{R}(G,\mu)$ is equal to $\bF_{s_1} \cup \partial \mathrm{T}_{2s_1}$.
\end{Exl}

The companion paper \cite{Woe+} deals mostly with \emph{full} versions of the ratio-limit compacta, which are generally different from the respective ones considered here\footnote{In \cite{Woe+}, the full ratio-limit space and boundary are referred to simply as ratio-limit compactification and boundary respectively, and our ratio-limit space and boundary are also refereed to as the \emph{reduced} ratio-limit compactification and boundary respectively.}. The full ratio-limit space is defined without incorporating $R_{\mu}$ into the picture, and is the smallest compactification $\Delta_{\mathrm{R}}G$ of $G$ to which the ratio-limit functions $y\mapsto H(x,y)$ extend continuously (see \cite[Section 6]{Woe+} for a comparison). It is straightforward to show that the quotient map $G \rightarrow G / R_{\mu}$ extends to a continuous surjective $G$-equivariant map from $\Delta_{\mathrm{R}}G$ onto $R(G,\mu)$. However, a key observation is that the full ratio-limit boundary and the (reduced) ratio-limit boundary considered in this paper coincide whenever $G$ is infinite and $R_{\mu}$ is finite. Hence, by \cite[Corollary 6.6]{Woe+} the two ratio-limit boundaries coincide for all classes of random walks considered in \cite{Woe+}.

A key step in the computation of full ratio limit boundaries in \cite{Woe+} is to show that they coincide with the respective $\rho$-Martin boundaries, whose computation was previously attained in many classes of examples. More precisely, for the classes of examples considered in \cite{Woe+}, it follows that the quotient map $G \rightarrow G / R_{\mu}$ induces a homeomorphism $\tau : \partial_{\Delta, \rho} G \rightarrow \partial_{\mathrm{R}} G$ (which also satisfies $K(x,\xi) = H(x,\tau(\xi))$ for every $x\in G$ and $\xi \in \partial_{\Delta, \rho}G$).

In such cases, a simple approximation argument together with continuity of left multiplication by $g$ shows that $\tau(g\xi) = g\tau(\xi)$ for every $\xi \in \partial_{\Delta, \rho}G$ and $g\in G$. Thus, we get that the identification $\tau$ is automatically $G$-equivariant. This, together with the above examples, suggests the following question:

\begin{Ques}
Suppose $P$ is a random walk on $G$ induced by a finitely supported measure $\mu$ with SRLP and spectral radius $\rho$. Does the $\rho$-Martin compactification cover the ratio-limit space? More precisely, is there a surjective $G$-equivariant continuous map $\tau :\Delta_{\rho}G \rightarrow \mathrm{R}(G,\mu)$ which restricts to the quotient map $G \rightarrow G / R_{\mu}$ on $G$?
\end{Ques}

%%%%%%%%%%%%%%%%%%%%%%%%%%%%%%%%
	
\section{Toeplitz quotient C*-algebras for random walks.} \label{S:C*-alg-for-Markov-chains}

The ratio-limit space $\mathrm{R}(G,\mu)$ arises from the computation of the Cuntz C*-algebra $\O(G,\mu)$ as part of its spectrum. In this section of the paper we will show this. Toeplitz C*-algebras, tensor algebras and C*-envelopes arising from stochastic matrices were studied previously in \cite{DOM14, DOM16} (see also \cite{CDHLZ+}), but the definition of Cuntz C*-algebra we give below is new.

Let $P$ be the stochastic matrix over a set $\mathcal{X}$. For each $m \in \bN$ we denote $\F^{(m)}_P$ the Hilbert space with orthonormal basis $\{e_{jk}^{(m)}\}_{(j,k)\in E(P^m)}$. The Fock Hilbert space of $P$ is then given by
$$
\F_P:= \bigoplus_{m=0}^{\infty} \F^{(m)}_P
$$

Next, for each $n\in \bN$ and $(i,j) \in E(P^n)$ we define an operator $S^{(n)}_{ij}$ on $\F_P$ by setting for every $(j',k) \in E(P^m)$,
$$
S^{(n)}_{ij}(e_{j'k}^{(m)}) = \delta_{jj'}\sqrt{\frac{P^{(n)}_{ij}P^{(m)}_{jk}}{P^{(n+m)}_{ik}}}e_{ik}^{(n+m)}.
$$
Since $S^{(n)}_{ij}$ maps an orthonomal basis to a uniformly bounded (by $1$) orthogonal set, it defines a bounded operator on $\F_P$. For a fixed $k\in \mathcal{X}$, we denote by $\F_{P,k}$ the closed linear span of $\{ \ e^{(m)}_{jk} \ | \ (j,k)\in E(P^m), \ m\geq 0 \ \}$. It follows that $\F_{P,k}$ is a reducing subspace for the operators $S^{(n)}_{ij}$. For a fixed $m\in \bN$ we also denote $\F^{(m)}_{P,k}$ the closed linear span of $\{ \ e^{(m)}_{jk} \ | \ (j,k)\in E(P^m) \ \}$.

\begin{Def}
Let $P$ be a stochastic matrix on a set $\mathcal{X}$. The Toeplitz C*-algebras of $P$ is given by
$$
\T(P):= C^*( \ S^{(n)}_{ij} \ | \ (i,j)\in E(P^n), \ n \in \bN \ ).
$$
\end{Def}

Note that $c_0(\mathcal{X}) \subseteq \T(P)$ via the identification $(c_i) \mapsto \sum_{i\in \mathcal{X}}c_i S^{(0)}_{ii}$ for $(c_i) \in c_0(\mathcal{X})$. We will henceforth identify $c_0(\mathcal{X})$ with its copy in $B(\F_P)$ as above, and denote by $p_i: = S^{(0)}_{ii}$ the operator corresponding in $c_0(\mathcal{X})$ to the characteristic function of $i \in \mathcal{X}$. 

\begin{Rmk}
We warn the reader that the Toeplitz C*-algebra $\T(P)$ defined here and in \cite{DOM14,DOM16} for a stochastic matrix $P$ are different when $\mathcal{X}$ is infinite. For instance, the former is non-unital while the latter is unital. In Section \ref{s:symmetry-uniqueness} we will see how the Toeplitz C*-algebras given here arise from subproduct systems with coefficients $c_0(\mathcal{X})$, while the Toeplitz C*-algebra in \cite{DOM14,DOM16} arise from subproduct systems with coefficients $\ell^{\infty}(\mathcal{X})$. 
\end{Rmk}

\begin{Def} \label{D:RW-cuntz-algebra}
Let $P$ be a stochastic matrix over $\mathcal{X}$. Denote by $\J(P):= \T(P) \cap \prod_{k\in \mathcal{X}} \mathbb{K}(\F_{P,k})$, which is a closed ideal in $\T(P)$. We define the Cuntz C*-algebra of $P$ to be
$$ 
\O(P) := \T(P) / \J(P).
$$
\end{Def}

We let $q_P : \T(P) \rightarrow \O(P)$ be the natural quotient map. Since for each $i\in P$ we have that $p_i = S^{(0)}_{ii} \notin \prod_{k\in P} \mathbb{K}(\F_{P,k})$, we see that $\{q_P(p_i)\}_{i\in \mathcal{X}}$ are still pairwise orthogonal projections, so that $q_P$ is injective on $c_0(\mathcal{X})$. Hence, we may also identify $c_0(\mathcal{X})$ as a subalgebra of $\O(P)$ via $q_P$. 

Henceforth, we will assume that $P$ is a random walk on a group $G$ induced by a finitely supported measure $\mu$. To emphasize this we denote 
$$
\T(G,\mu) := \T(P), \ \J(G,\mu):= \J(P), \ \text{and} \ \O(G,\mu) := \O(P).
$$

For $m\in \bN$ and $x\in G$, denote by $Q^{(m)}$ the orthogonal projection from $\F_P$ onto $\F^{(m)}_P$, and $Q^{(m)}_x := Q^{(m)} p _x = p_x Q^{(m)}$. Denote also $Q^{[m,\infty)} := \sum_{\ell=m}^{\infty} Q^{(\ell)}$, and $Q^{[m,\infty)}_x:= Q^{[m,\infty)}p_x = p_x Q^{[m,\infty)} = \sum_{\ell=m}^{\infty} Q^{(\ell)}_x$.

\begin{Prop} \label{p:viselter-ideal}
Let $P$ be a random walk on a group $G$ induced by a finitely supported measure $\mu$. Then $Q^{(0)}_x \in \T(G,\mu)$ for every $x \in G$. Moreover, we have that the closed ideal $\I_{\mathbb{K}}:= \lip Q^{(0)}_z \rip_{z\in G} \lhd \T(G,\mu)$ is equal to $\oplus_{z \in G} \mathbb{K}(\F_{P,z})$, and that $Q_x^{(\ell)} \in \I_{\mathbb{K}}$ for every $\ell \in \bN$ and $x \in G$.
\end{Prop}

\begin{proof}
Since $\Gr(P)$ is locally finite, for every $x\in G$ there are finitely many $y \in G$ such that $(x,y) \in E(P)$. Hence, for every $x\in G$ we have that
$$
R^{(0)}_x:= S^{(0)}_{x,x} - \sum_{(x,y)\in E(P)} S^{(1)}_{x,y}S^{(1)*}_{x,y} \in \T(G,\mu).
$$
Then, on a standard basis vector $e^{(m+1)}_{x,z}$ for $m \in \bN$ we get
$$
R^{(0)}_x(e^{(m+1)}_{x,z}) = e^{(m+1)}_{x,z} - \sum_{y\in G}\frac{P^{(n)}_{x,y}P^{(m)}_{y,z}}{P^{(n+m)}_{x,z}}e^{(m+1)}_{x,z} = 0,
$$
and since $R^{(0)}_x(e^{(0)}_{x,x}) = e^{(0)}_{x,x}$ for each $x\in G$, and $R^{(0)}_x(e^{(m)}_{y,z}) = 0$ if $x\neq y$, we get that $Q^{(0)}_x= R^{(0)}_x \in \T(G,\mu)$.

Since $\{Q^{(0)}_z\}_{z\in G}$ is a set of pairwise orthogonal rank-one projection, each onto the subspace $\mathbb{C}e^{(0)}_{z,z} \subseteq \F_{P,z}$, and since $\T(G,\mu) e^{(0)}_{z,z} = \F_{P,z}$, we see that the closed ideal $\lip Q^{(0)}_z \rip_{z\in G}$ is equal to $\oplus_{z \in G} \mathbb{K}(\F_{P,z})$.

Finally, since $\Gr(P^{\ell})$ is locally finite, for $x\in G$ we see that $Q^{(\ell)}_x: = p_x Q^{(\ell)} = Q^{(\ell)}p_x$ is finite rank, and $Q^{(\ell)}_x = \sum_{(x,z) \in E(P^{\ell})}Q^{(\ell)}_{x,z}$ where $Q^{(\ell)}_{x,z}$ is the rank-one projection onto the subspace $\mathbb{C}e_{x,z}^{(\ell)}$. Since $Q^{(\ell)}_{x,z} \in \mathbb{K}(\F_{P,z})$ for each $x,z\in G$ and $\ell \in \bN$, the proof is concluded.
\end{proof}

\begin{Rmk}
In Proposition \ref{p:viselter-subprod-ideal} we will see that $\I_{\mathbb{K}}$ coincides with Viselter's ideal of $\T(G,\mu)$, which is realized as the Toeplitz C*-algebra of a subproduct system arising from the random walk as in Section \ref{s:symmetry-uniqueness}.
\end{Rmk}

We will define an auxiliary C*-algebra $\widehat{\T}(G,\mu)$ and auxiliary operators $\{W^{(n)}_{x,y}\}$ and $\{T^{(n)}_{x,y}\}$ which will help make our computation of $\O(G,\mu)$ easier. Denote by $\J_{\mathbb{K}}:= \prod_{z \in G} \mathbb{K}(\F_{P,z})$, and let
$$
\widehat{\T}(G,\mu):= \T(G,\mu) + \J_{\mathbb{K}}.
$$
Since $\J(G,\mu) = \T(G,\mu) \cap \J_{\mathbb{K}}$ by definition, by \cite[Corollary I.5.6]{Davbook} we get that
$$
\bigslant{\widehat{\T}(G,\mu)}{[\J(G,\mu) + \J_{\mathbb{K}}]} \cong \O(G,\mu).
$$
Hence, even though some operators we define may not be in $\T(G,\mu)$, they will all be in $\widehat{\T}(G,\mu)$ so that their images in $\O(G,\mu)$ will make sense. More precisely, $q_P : \T(G,\mu) \rightarrow \O(G,\mu)$ extends to a well-defined quotient map (denoted still by) $q_P: \widehat{\T}(G,\mu) \rightarrow \O(G,\mu)$, and we denote for an operator $T\in \widehat{\T}(G,\mu)$ its image in $\O(G,\mu)$ by $\overline{T} := q_P(T)$.

\begin{Prop} \label{p:norm-computation}
Let $P$ be a random walk on a group $G$ induced by a finitely supported measure $\mu$. Then, for any $T\in \widehat{\T}(G,\mu)$ we have that
$$
\|\overline{T} \| = \sup_{z\in G} \lim_m \| T Q^{[m,\infty)}|_{\F_{P,z}} \|
$$
\end{Prop}

\begin{proof}
For every $\epsilon > 0$ there is some $K \in \J_{\mathbb{K}}$ such that 
\begin{align*}
\|\overline{T} \| \geq \|T + K \| - \epsilon = \sup_{z\in G} \| [T + K]|_{\F_{P,z}} \| - \epsilon \geq \\
\sup_{z\in G} \lim_m \| [T + K]Q^{[m,\infty)}|_{\F_{P,z}}\| - \epsilon 
\end{align*}
But since for every $m\in \bN$ and $z\in G$ we have
$$
\| [T + K]Q^{[m,\infty)}|_{\F_{P,z}}\| \geq 
\| TQ^{[m,\infty)}|_{\F_{P,z}}\| - \| KQ^{[m,\infty)}|_{\F_{P,z}}\|,
$$
by taking $m\rightarrow \infty$, and as $K|_{\F_{P,z}} \in \mathbb{K}(\F_{P,z})$ for $z\in G$, we get that
\begin{equation*}
\|\overline{T} \| \geq \sup_{z\in G} \lim_m \| TQ^{[m,\infty)}|_{\F_{P,z}}\| - \epsilon.
\end{equation*}
Hence, we arrive at the lower bound $\|\overline{T} \| \geq \sup_{z\in G} \lim_m \| TQ^{[m,\infty)}|_{\F_{P,z}}\|$.

On the other hand, for every $T \in \widehat{\T}(G,\mu)$ and a sequence of natural numbers $(m_z)_{z\in G}$, we get by local finiteness of $\Gr(P)$ that the operator $T|_{\F_{P,z}} \cdot (I - Q^{[m_z,\infty)})|_{\F_{P,z}}$ on $\F_{P,z}$ is finite rank, so that
$$
T_0:= \underset{z\in G}{\oplus}\Big[T|_{\F_{P,z}} \cdot (I - Q^{[m_z,\infty)})|_{\F_{P,z}} \Big] \in \J_{\mathbb{K}} = \prod \mathbb{K}(\F_{P,z}).
$$
Thus, we get for any sequence of natural numbers $(m_z)_{z\in G}$ that
$$
\| \overline{T} \| \leq \| T - T_0 \| = \sup_{z\in G} \| TQ^{[m_z,\infty)}|_{\F_{P,z}} \|.
$$
Since $(m_z)_{z\in G}$ is arbitrary, we get the upper bound
$$
\| \overline{T} \| \leq \sup_{z\in G} \lim_m \| TQ^{[m,\infty)}|_{\F_{P,z}} \|.
$$
Combined with the the previously obtained lower, we get our result.
\end{proof}

Suppose that $P$ is a random walk with SRLP on a group $G$ induced by a finitely supported measure $\mu$. Then, for any $n\in \bN$ and $(x,y) \in E(P^n)$ we define two operators $W_{x,y}^{(n)}$ and $T_{x,y}^{(n)}$ on $\F_P$ by setting for $(y',z) \in E(P^m)$,
$$
W_{x,y}^{(n)}(e_{y',z}^{(m)}) = \delta_{y,y'} \cdot \sqrt{H(x^{-1}y,x^{-1}z)} \cdot e_{x,z}^{(m+n)}, 
$$ 
and $T_{x,y}^{(n)}:= \Big[ \frac{\rho(P)}{P^{(n)}_{x,y}}\Big]^{\frac{n}{2}} S^{(n)}_{x,y} \in \T(P)$, alternatively given by the formula 
$$
T^{(n)}_{x,y}(e^{(m)}_{y',z}) = \delta_{y,y'} \sqrt{\frac{\rho(P)^nP^{(m)}_{y,z}}{P^{(n+m)}_{x,z}}}e^{(n+m)}_{x,z}.
$$ 
Boundedness of the operators $W_{x,y}^{(n)}$ and $T_{x,y}^{(n)}$ can be observed from the estimates in Section \ref{s:ratio-limit space} and the fact that the ratio-limit functions are bounded. Then, their adjoints are given for $x',z \in G$ by $W^{(n)*}_{x,y}(e^{(m)}_{x',z}) = T^{(n)*}_{x,y}(e^{(m)}_{x',z}) = 0$ for $m < n$ and otherwise for $m\in \bN$ we have
$$
W^{(n)*}_{x,y} (e^{(n+m)}_{x',z}) = \begin{cases} 
\delta_{x,x'} \cdot \sqrt{H(x^{-1}y,x^{-1}z)} \cdot e^{(m)}_{y,z} &\mbox{if } (y,z)\in E(P^m) \\
0 & \mbox{if } \ \text{otherwise}. \end{cases}
$$
and
$$
T^{(n)*}_{x,y}(e^{(n+m)}_{x',z}) = \delta_{x,x'} \sqrt{\frac{\rho(P)^nP^{(m)}_{y,z}}{P^{(n+m)}_{x,z}}}e^{(m)}_{y,z}.
$$

\begin{Prop} \label{p:indep-of-m}
Suppose $P$ is a random walk on a group $G$ induced by a finitely supported measure $\mu$, and assume $P$ has SRLP. Then for every $n\in \bN$ and $(x,y)\in E(P^n)$ we have that $T^{(n)}_{x,y} - W^{(n)}_{x,y}\in \J_{\mathbb{K}}$. In particular, we get that $W^{(n)}_{x,y} \in \widehat{\T}(G,\mu)$.
\end{Prop}

\begin{proof}
Fix $z \in G$. It will suffice to show that the restriction of $T^{(n)}_{x,y} - W^{(n)}_{x,y}$ to $\F_{P,k}$ is compact. Let $m\in \bN$. If $(y,z) \notin E(P^m)$, then $T^{(n)}_{x,y} - W^{(n)}_{x,y}$ is zero on $\F_{P,z}^{(m)}$. If $(y,z) \in E(P^m)$, then $e^{(m)}_{y,z}$ is the only standard basis vector of $\F^{(m)}_P$ which is not annihilated by $T^{(n)}_{x,y} - W^{(n)}_{x,y}$. In this case, we get that
$$
\|[T^{(n)}_{x,y} - W^{(n)}_{x,y}](e^{(m)}_{y,z}) \| = \Big{|}\sqrt{\frac{\rho(P)^nP^{(m)}_{y,z}}{P^{(n+m)}_{x,z}}} - \sqrt{H(x^{-1}y,x^{-1}z)}\Big{|}.
$$
However, since $T^{(n)}_{x,y} - W^{(n)}_{x,y}$ is at most a rank-one operator when restricted to an operator from $\F^{(m)}_{P,z}$ to $\F^{(m+n)}_{P,z}$, it will suffice to show that as $m \rightarrow \infty$, the above goes to $0$. But now, the estimates in Section \ref{s:ratio-limit space} (up to applying a square root) establish this convergence.
\end{proof}

\begin{Rmk}
It is at this point where we see the importance of defining $\O(G,\mu)$ as a quotient by $\J_{\mathbb{K}} \cap \T(G,\mu)$ as opposed to a quotient by $\I_{\mathbb{K}} = \oplus_{z\in G} \mathbb{K}(\F_{P,z}) \lhd \T(G,\mu)$. It turns out that in most cases $\O(G,\mu)$ is a proper quotient of $\T(G,\mu) / \I_{\mathbb{K}}$. This is because of the following reasoning.

When $G$ is infinite, since $\mu$ is finitely supported, for each $x,y \in G$ and $m \in \bN$ we may always choose $z$ for which $P^{(m)}_{x,z} = P^{(m)}_{y,z} = 0$. Hence, we see that the convergence 
$$
\frac{\rho(P)^nP^{(m)}_{y,z}}{P^{(n+m)}_{x,z}} \underset{m\rightarrow \infty}{\longrightarrow} H(x^{-1}y,x^{-1}z)
$$ 
is never uniform in $z$, and we get that $T^{(n)}_{x,y} - W^{(n)}_{x,y} \notin \I_{\mathbb{K}}$. On the other hand we have shown above that $T^{(n)}_{x,y} - W^{(n)}_{x,y} \in \J_{\mathbb{K}}$. Thus, in order to show a proper inclusion $\I_{\mathbb{K}} \subsetneq \J(G,\mu)$, it will suffice to show that $W^{(n)}_{x,y} \in \T(G,\mu)$, so that $T^{(n)}_{x,y} - W^{(n)}_{x,y}$ is in $\J(G,\mu) = \T(G,\mu) \cap \J_{\mathbb{K}}$ but not in $\I_{\mathbb{K}}$. 

This can be done for instance when $R_{\mu} = G$, so that all ratio-limit functions $\{H(x,\cdot)\}_{x\in G}$ are constant $1$. Indeed, one can show that $W^{(n)}_{x,y}$ is the partial isometry in the polar decomposition $T^{(n)}_{x,y} = W^{(n)}_{x,y} A$ where $A \in \T(G,\mu)$ is positive with $\sigma(A)$ bounded away from $0$. Continuous functional calculus can then used to show $W^{(n)}_{x,y} \in \T(G,\mu)$, with similar techniques as below.
\end{Rmk}

Next, for $(x,y) \in E(P^n)$ we denote $R_{x,y}:= R^{(n)}_{x,y} = \sqrt{W_{x,y}^{(n)*}W_{x,y}^{(n)}} \in \widehat{\T}(G,\mu)$. By definition, we get for $(y',z) \in E(P^m)$ that,
$$
R_{x,y}(e^{(m)}_{y',z}) = \delta_{y,y'} \cdot H(x^{-1}y,x^{-1}z)^{\frac{1}{2}} \cdot e^{(m)}_{y,z}.
$$
But now, since $(x,y)\in E(P^n)$ are fixed, by estimates in Section \ref{s:ratio-limit space} there are $c_{x,y}, C_{x,y} >0$ such that $0< c_{x,y} \leq H(x^{-1}y,x^{-1}z) \leq C_{x,y} < \infty$ for all $z\in G$. Hence, we get that $\sigma(R_{x,y}) \subseteq [c^{1/2}_{x,y}, C^{1/2}_{x,y}]$, and by applying the non-negative continuous function 
$$
t\mapsto \begin{cases}
0 & t\in (-\infty,0) \\
t \cdot c_{x,y}^{-1} & t \in [0, c^{1/2}_{x,y}] \\
t^{-1} & t\in [c^{1/2}_{x,y}, C^{1/2}_{x,y}] \\
C_{x,y}^{-1/2} & t \in (C^{1/2}_{x,y},\infty)
\end{cases}
$$ 
to the positive operator $R_{x,y}$, we get the positive operator $R_{x,y}' \in \widehat{\T}(G,\mu)$ given for $(y',z) \in E(P^m)$ by
$$
R_{x,y}'(e_{y',z}^{(m)}) = \delta_{y,y'} \cdot H(x^{-1}y,x^{-1}z)^{-\frac{1}{2}} \cdot e_{y,z}^{(m)}.
$$
But then, $V^{(n)}_{x,y} = W^{(n)}_{x,y} R_y' \in \widehat{\T}(G,\mu)$ is given by
$$
V_{x,y}^{(n)}(e_{y',z}^{(m)}) = \delta_{y,y'} e_{x,z}^{(m+n)}. 
$$

Now fix $x,y,z \in G$. Since $P$ has SRLP it must be aperiodic, so there exists $n_0$ (depending on $x,y$ and $z$) such that $(x,x), (x,y), (y,y),(y,z), (z,z) \in E(P^{n})$ for all $n \geq n_0$. Thus, we may define the following operators:
\begin{enumerate}
\item $E_{x,y} = V_{x,x}^{(n)*}V_{x,y}^{(n)}$

\item $U_x = V_{x,x}^{(n)*}V_{x,x}^{(n+1)}$, and let $U = \oplus_{x\in G} U_x$. 

\item $H^{(z)}_{x,y} = E_{z,y} R^{(n)}_{x,y}E_{y,z}$, and let $H_{x,y}: = \oplus_{z\in G} H^{(z)}_{x,y}$.
\end{enumerate}

It is readily verified that the definitions of $E_{x,y}$, $U_x$ and $H^{(z)}_{x,y}$ are independent of $n\geq n_0$ modulo $\J_{\mathbb{K}}$, by showing that the the restrictions to $\F_{P,z}$ of differences (with different values of $n \geq n_0$) are in $\mathbb{K}(\F_{P,z})$ for each $z\in G$.

For an operator $T \in \prod_{z\in G} B(\F_{P,z})$ we denote by $\overline{T}$ its image in the Calkin quotient $\prod_{z\in G} B(\F_{P,z}) / \prod_{z\in G} \mathbb{K}(\F_{P,z}) \cong \prod_{z\in G} \Q(\F_{P,z})$, so that when $T\in \widehat{\T}(G,\mu)$ we have that $\overline{T} \in \O(G,\mu)$.

\begin{Prop} \label{P:relations-up-to-cpts}
Let $P$ be a random walk on a group induced by a finitely supported $\mu$, and assume $P$ has SRLP. Then, 
\begin{enumerate}
\item the family of operators $\{\overline{E}_{x,y}\}$ is a $G\times G$ system of matrix units.

\item the family $\{\overline{E}_{x,y}\}$ commutes with $\{\overline{H}_{x,y}\}$ and $\overline{U}$.

\item for each $x, y \in G$ we have $\overline{H}_{x,y}\overline{U} = \overline{U} \overline{H}_{x,y}$.

\item $\overline{U}$ is a unitary element, and each $\overline{U}_x$ has spectrum $\sigma(\overline{U}_x) = \mathbb{T} \cup \{0\}$.

\item $\O(G,\mu)$ is generated by $\{\overline{E}_{x,y}\}_{x,y\in G}$, $\{\overline{H}^{(e)}_{x,y}\}_{x,y \in G}$ and $\overline{U}_e$.

\end{enumerate}
\end{Prop}

\begin{proof}
We first show $(1)$. Let $x,y,y',z \in G$. Then, by aperiodicity of $P$, for fixed $w\in G$ there is $m_0$ large enough so that $(y,w), (x,w), (x,x), (y',y') \in E(P^m)$ for $m\geq m_0$. Hence, whenever $(z',w)\in E(P^m)$ and $m\geq m_0$ we have,
$$
E_{x,y} E_{y',z}(e^{(m)}_{z',w}) = \delta_{z,z'} E_{x,y}(e^{(m)}_{y',w}) = \delta_{y,y'} \delta_{z,z'} e^{(m)}_{x,w} = \delta_{y,y'} E_{x,z}(e^{(m)}_{z',w}).
$$
Hence, we get that $E_{x,y} E_{y',z} - E_{x,z} \in \J_{\mathbb{K}}$. A similar computation shows that $E_{x,y}^* - E_{y,x} \in \J_{\mathbb{K}}$ as well. Hence, $\{\overline{E}_{x,y}\}$ is a $G\times G$ system of matrix units.

Next, we show $(2)$. Indeed, by item $(1)$ we have for $x,x',y,y' \in G$ that
$$
\overline{E}_{x,y} \overline{H}_{x',y'} = \overline{E}_{x,y'}\overline{R}_{x',y'}\overline{E}_{y',y} = \overline{H}_{x',y'}\overline{E}_{x,y}.
$$
To show that $\{\overline{E}_{x,y}\}$ commutes with $\overline{U}$ it will suffice to show that $E_{x,y}U_y - U_x E_{x,y} \in \J_{\mathbb{K}}$. So, for fixed $z \in G$, by aperiodicity of $P$ there is $m_0$ large enough so that $(x,z), (y,z) \in E(P^m)$ for all $m\geq m_0$. Hence, whenever $(y',z) \in E(P^m)$ for $m \geq m_0$ we have
$$
E_{x,y}U_y(e^{(m)}_{y',z}) = \delta_{y,y'} e^{(m+1)}_{x,z} = U_x E_{x,y}(e^{(m)}_{y',z}).
$$

Now, we show item $(3)$. By aperiodicity of $P$, for fixed $z\in G$ there is $m_0$ large enough so that $(y',z) \in E(P^m)$ for all $m\geq m_0$, so that
$$
H_{x,y}U(e^{(m)}_{y',z}) = \sqrt{H(x^{-1}y,x^{-1}z)} \cdot e^{(m+1)}_{y',z} = U H_{x,y}(e^{(m)}_{y',z}).
$$
Thus, we get that $H_{x,y}U - UH_{x,y} \in \J_{\mathbb{K}}$.

To show item $(4)$, fix $z\in G$, so that by aperiodicity there is $m_0$ large enough so that $(y,z) \in E(P^m)$ for all $m\geq m_0$. Hence, for any $m\geq m_0+1$ and $(y,z) \in E(P^m)$ we have
$$
U^*U(e^{(m)}_{y,z}) = e^{(m)}_{y,z} = UU^*(e^{(m)}_{y,z}).
$$
Thus, we get that $U^*U - I, UU^* - I \in \J_{\mathbb{K}}$. Since $U_y$ acts as the unilateral shift on the orthonormal set $\{e_{y,x}^{(m)}\}_{m\geq m_0}$, it follows that $\mathbb{D} \subseteq \sigma(U_y)$. Since $U_y$ is the compression of an essential unitary to one of its reducing subspaces, it must be a normal partial isometry, and we get that $\sigma(\overline{U}_y) = \mathbb{T} \cup \{0\}$.

Finally, we show item $(5)$. First note that by construction the operators $\{\overline{E}_{x,y}\}_{x,y\in G}$, $\{\overline{H}^{(e)}_{x,y}\}_{x,y \in G}$ and $\overline{U}_e$ are indeed in $\O(G,\mu)$. To show that these operators generate $\O(G,\mu)$ as a C*-algebra, first note that by Proposition \ref{p:indep-of-m} we have that $\{\overline{W}_{x,y}\}$ are generators for $\O(G,\mu)$. Then, it will suffice to establish for $x,y \in G$ and $n\in \bN$ that
$$
\overline{W}^{(n)}_{x,y} = \overline{V}_{x,y}^{(n)} \overline{R}_{x,y} = \overline{U}_x^n \overline{E}_{x,e}\overline{H}_{x,y}^{(e)} \overline{E}_{e,y}.
$$
So, for a fixed $z\in G$, by aperiodicity of $P$ there is $m_0$ large enough so that $(e,z), (y,z), (x,z) \in E(P^m)$ for all $m\geq m_0$. Hence, for $(y',z) \in E(P^m)$ and $m\geq m_0$ we have
$$
W^{(n)}_{x,y}(e^{(m)}_{y',z}) = \delta_{y,y'} \cdot \sqrt{H(x^{-1}y,x^{-1}z)} \cdot e^{(m+n)}_{x,z} = 
$$
$$
U_x^n \big( \delta_{y,y'} \cdot \sqrt{H(x^{-1}y,x^{-1}z)} \cdot e^{(m)}_{x,z} \big) = U_x^n E_{x,e} H^{(e)}_{x,y} E_{e,y} (e^{(m)}_{y',z}).
$$
Thus, we see that $W^{(n)}_{x,y} - U_x^n E_{x,e} H^{(e)}_{x,y} E_{e,y} \in \J_{\mathbb{K}}$, and the proof is concluded.
\end{proof}

Recall that $Q^{(m)}$ denotes the orthgonal projection from $\F_P$ onto $\F_P^{(m)}$, and that $Q^{[m,\infty)} := \sum_{\ell=m}^{\infty} Q^{(\ell)}$ is the projection from $\F_P$ onto $\oplus_{\ell=m}^{\infty} \F_P^{(\ell)}$.

\begin{Thm} \label{t:algebra-computation}
Let $P$ be a random walk on a group $G$ induced by a finitely supported measure $\mu$, and assume $P$ has SRLP. Then, 
$$
\O(G,\mu) \cong  C(\mathrm{R}(G,\mu) \times \mathbb{T}) \otimes \mathbb{K}(\ell^2(G)).
$$
\end{Thm}

\begin{proof}
By item $(5)$ of Proposition \ref{P:relations-up-to-cpts} we know that $\O(G,\mu)$ is generated by $\{\overline{E}_{x,y}\}_{x,y\in G}$, $\{\overline{H}^{(e)}_{x,y}\}_{x,y \in G}$ and $\overline{U}_e$. 

By items $(1)$ and $(2)$ of Proposition \ref{P:relations-up-to-cpts} the operators $\{\overline{E}_{x,y}\}_{x,y\in G}$ form a system of matrix units which commute with the self-adjoint operators $\{\overline{H}_{x,y} \}$ and $\overline{U}$. Hence, we get that $\O(G,\mu) \cong \A \otimes \mathbb{K}(\ell^2(G))$ where $\A$ is the corner C*-algebra generated by $\{\overline{H}^{(e)}_{x,y}\}_{x,y \in G}$ together with $\overline{U}_e$.

By items $(3)$ and $(4)$ of Proposition \ref{P:relations-up-to-cpts} we get that $\overline{U}_e$ is a unitary element of $\A$ which commutes with the self-adjoint elements $\overline{H}^{(e)}_{x,y}$ for every $x,y \in G$, and that $\sigma(\overline{U}_e) = \mathbb{T}$ (as an element in $\A$). Hence, we get that $\A = C(X)$ is commutative, with spectrum $X = \mathbb{T} \times Y$ so that $Y$ is the spectrum of the unital commutative C*-algebra $C(Y)$ generated by $\overline{H}^{(e)}_{x,y}$ for $x,y\in G$.

Denote by $d_{x,y}$ the function given by $d_{x,y}(z) = \sqrt{H(x^{-1}y,x^{-1} z)}$. Then, the rule $\varphi(\overline{H}^{(e)}_{x,y}) = d_{x,y}$ extends to a $*$-isomorphism $\varphi : C(Y) \rightarrow C(\mathrm{R}(G,\mu))$. Indeed, if $\overline{T}:= \sum_{i=1}^n c_i \overline{M}_i \in C(Y)$ is a finite linear combination of monomials in (self-adjoint) generators $\{\overline{H}_{x,y}^{(e)}\}$, where $M_i = \prod_{j=1}^{\ell_i} H^{(e)}_{x_{i,j},y_{i,j}}$, then its norm as an element in $\O(G,\mu)$ is given by Proposition \ref{p:norm-computation} as
$$
\|\overline{T}\| = \sup_{z\in G} \lim_m \| TQ^{[m,\infty)}|_{\F_{P,z}} \| = 
$$
$$
\sup_{z\in G} \lim_m \Big{\|} \Big[\sum_{i=1}^n c_i M_i \Big] (e^{(m)}_{e,z}) \Big{\|} = \sup_{z\in G} \big| \sum_{i=1}^n c_i \cdot \prod_{j=1}^{\ell_i}d_{x_{i,j},y_{i,j}}(z) \big|,
$$
where the second and third equalities hold because $TQ^{[m,\infty)}|_{\F_{P,z}}$ is a diagonal operator with eigenvectors $e^{(m)}_{e,z}$ for $(e,z) \in E(P^m)$ whose eigenvalues are independent of $m$. Thus, by Stone--Weierstrass theorem together with the fact that $d_{x,y}$ separate points in $\mathrm{R}(G,\mu)$, we get that $\varphi$ extends to a $*$-isomorphism.
\end{proof}

\begin{Rmk}
Since $\O(G,\mu)$ can be defined without assuming SRLP, one may ask whether some compact $G$-space appears in its computation without the assumption of SRLP, as $\mathrm{R}(G,\mu)$ does in the presence of SRLP. This seems to be the case under certain mild assumptions on the random walk, and provides a generalized notion of the ratio-limit space without assuming SRLP. We thank Guy Salomon for raising this question.
\end{Rmk}

As a consequence of our computation of $\O(G,\mu)$, we obtain the following simple corollary. Recall the definition of the ratio-limit radical $R_{\mu} \leq G$ of Proposition \ref{p:rlr}, in the presence of SRLP.

\begin{Cor}
Let $P$ be a random walk on a group $G$ induced by a finitely supported measure $\mu$, and assume $P$ has SRLP. Then the primitive ideal spectrum of $\O(G,\mu)$ is homeomorphic to $\mathbb{T}$ if and only if $R_{\mu} = G$.
\end{Cor}

\begin{proof}
First note that the primitive ideal space of $\O(G,\mu)$ is homeomorphic to $\mathrm{R}(G,\mu) \times \bT$ by Theorem \ref{t:algebra-computation}.

If $G=R_{\mu}$, then we get that $H(x,y)$ are constant in $y$ by definition, so that the ratio limit space $\mathrm{R}(G,\mu)$ is trivial. Hence, we get that $\O(G,\mu)$ has primitive ideal spectrum homeomorphic to $\bT$.

Conversely, if $\phi: \bT \rightarrow \mathrm{R}(G,\mu) \times \bT$ is a homeomorphism, let $\id \times \pi : \mathrm{R}(G,\mu) \times \bT \rightarrow \mathrm{R}(G,\mu)$ be the projection onto the first coordinate. Then, since $\bT$ is connected, so too would be $\mathrm{R}(G,\mu)$ as its image under $(\id \times \pi) \circ \phi$. However, $\mathrm{R}(G,\mu)$ contains the discrete subspace $G / R_{\mu}$, so that $\mathrm{R}(G,\mu)$ is connected if and only if $G / R_{\mu}$ is a single point, in which case $R_{\mu} = G$.
\end{proof}

%%%%%%%%%%%%%%%%%%%%%%%%%%%%%%%%%%%%%%%%%%%%%%%%%%%%%%%%%%%%%

\section{Symmetry-uniqueness and subproduct systems.} \label{s:symmetry-uniqueness}

In this section we show that when the $G$ action on the ratio limit boundary is minimal, there is a unique smallest quotient of $\T(G,\mu)$ that respects natural $G\times \mathbb{T}$ symmetries coming from the random walk. After this is done, we explain how our C*-algebras arise from subproduct systems, and how this sheds light on Viselter's question in that context.

Let $P$ be a random walk on $G$ induced by a finitely supported measure $\mu$. The standard gauge action by the unit circle is the point-norm continuous action $\gamma : \mathbb{T} \rightarrow \Aut(\T(G,\mu))$ given by $\gamma_{\zeta}(T) = U_{\zeta} T U_{\zeta}^{-1}$ where $U_{\zeta} : \F_P \rightarrow \F_P$ is the unitary defined by $U_{\zeta}(e_{y,z}^{(m)}) = \zeta^m e_{y,z}^{(m)}$ for every $(y,z) \in E(P^m)$.

When $\mathrm{R}(G,\mu)$ is trivial, it readily follows that $\O(G,\mu) \cong C(\mathbb{T}) \otimes \mathbb{K}(\ell^2(G))$ is the unique smallest $\mathbb{T}$-equivariant quotient of $\T(G,\mu)$. On the other hand, when $\mathrm{R}(G,\mu)$ is non-trivial, the action of $\mathbb{T}$ on $\O(G,\mu) \cong C(\mathrm{R}(G,\mu) \times \mathbb{T}) \otimes \mathbb{K}(\ell^2(G))$ has at least two maximal $\mathbb{T}$ invariant proper ideals, so there is no unique smallest $\mathbb{T}$-equivariant quotient. 

Thus, in order to get unique smallest symmetry-equivariant quotients, we add additional symmetries to $\T(G,\mu)$ coming from $G$. For each $g\in G$ we define a unitary operator $V_g : \F_P \rightarrow \F_P$ given by $V_g(e^{(m)}_{x,y}) = e^{(m)}_{gx,gy}$. A computation then shows that for any $(x,y)\in E(P^n)$ we have $V_gS^{(n)}_{x,y} = S^{(n)}_{gx,gy} V_g$, so we then get an induced action $\delta: G \rightarrow \Aut(\T(G,\mu))$ given by $\delta_g(T) = V_g TV_g^{-1}$.

It is clear that $U_{\zeta} V_g = V_g U_{\zeta}$ for every $\zeta \in \mathbb{T}$ and $g\in G$, and we denote this unitary operator by $W_{g,\zeta}$. Hence, the actions $\gamma$ and $\delta$ commute and induce a point-norm continuous action $\lambda : G\times \mathbb{T} \rightarrow \Aut(\T(G,\mu))$ given by 
$$
\lambda_{(g,\zeta)}(T) = W_{g,\zeta}T W_{g,\zeta}^{-1}.
$$

Our goal in this section is to show that when the action of $G$ on $\partial_{\mathrm{R}} G$ is minimal, there is a unique largest $\lambda$-invariant proper ideal $\J_{\lambda}$ in $\T(G,\mu)$.

It is then clear that the quotient map $q_{\lambda}$ of $\T(G,\mu)$ by this ideal is automatically injective on $c_0(G) \subseteq \T(G,\mu)$, and hence this will establish a $G\times \mathbb{T}$-invariance uniqueness theorem for $\T(G,\mu) / \J_{\lambda}$.

Recall that $\J_{\mathbb{K}} := \prod_{z\in G}\mathbb{K}(\F_{P,z})$ is an ideal of $\widehat{\T}(G,\mu) : = \T(G,\mu) + \J_{\mathbb{K}}$ giving rise to the quotient $\O(G,\mu)$. It is easily shown that $V_g\J_{\mathbb{K}} V_g^{-1}= U_{\zeta}\J_{\mathbb{K}}U_{\zeta}^{-1} = \J_{\mathbb{K}}$ for $g\in G$ and $\zeta \in \mathbb{T}$, so that $\lambda$ extends to a point-norm continuous action (denoted still by) $\lambda : G \times \mathbb{T} \rightarrow \Aut(\widehat{\T}(G,\mu))$, making $\J_{\mathbb{K}}$ into a $\lambda$-invariant ideal of $\widehat{\T}(G,\mu)$. Hence, we obtain an induced action $\overline{\lambda} : G \times \mathbb{T} \rightarrow \Aut(\O(G,\mu))$ on the quotient. 

Since for $x,y\in G$ we have $\lambda_{g,\zeta}(W^{(n)}_{x,y}) = \zeta^n \cdot W^{(n)}_{gx,gy}$ and $\lambda_{g,\zeta}(V^{(n)}_{x,y}) = \zeta^n \cdot V^{(n)}_{gx,gy}$, it follows that $\lambda$ acts on generators of $\O(G,\mu)$ by,
$$
\overline{\lambda}_{g,\zeta}(\overline{E}_{x,y}) = \overline{E}_{gx,gy}, \ \ \overline{\lambda}_{g,\zeta}(\overline{H}^{(h)}_{x,y}) = \overline{H}^{(gh)}_{gx,gy}, \ \ \text{and} \ \ \overline{\lambda}_{g,\zeta}(\overline{U}_x) = \zeta \cdot \overline{U}_{gx},
$$
for $g\in G$ and $\zeta \in \mathbb{T}$. Let $\{e_g\}$ be a standard orthonormal basis for $\ell^2(G)$, and let $S_g \in B(\ell^2(G))$ be the unitary shift operator given by $S_g(e_h) = e_{gh}$. 

Recall now from Section \ref{s:ratio-limit space} that the compacta $\mathrm{R}(G,\mu)$ and $\partial_{\mathrm{R}} G$ carry a $G$ action induced from left multiplication on $G$, which gives rise to an action $\widehat{\beta}$ of $G$ on $C(\mathrm{R}(G,\mu))$ and $C(\partial_{\mathrm{R}} G)$ given by $\widehat{\beta}_g(f)(\alpha) = f(g^{-1} \alpha)$. Under the identification of $H_{x,y}^{(e)}$ with $d_{x,y}$ and of $\overline{E}_{x,y}$ as matrix units acting on $\ell^2(G)$ in Theorem \ref{t:algebra-computation}, it is readily verified that
$$
\overline{\lambda}_{g,\zeta}(f\otimes K)(\alpha,\xi) = f(g^{-1}\alpha,\zeta \xi) \otimes S_gK S_g^{-1}.
$$
Finally, recall the notation for the natural quotient map $q_P : \T(G,\mu) \rightarrow \O(G,\mu)$ by the ideal $\J(G,\mu):= \J_{\mathbb{K}} \cap \T(G,\mu)$. By the above, we see that $q_P$ is naturally a $G \times \bT$-equivariant map with the appropriate $G\times \mathbb{T}$ actions. Hence, the ideal
$$
\J_{\lambda} := q_P^{-1}\big[(C([G / R_{\mu}] \times \mathbb{T}) \otimes \mathbb{K}(\ell^2(G)) \big],
$$
is clearly $\lambda$-invariant in $\T(G,\mu)$, and is proper if and only if $\partial_{\mathrm{R}}G \neq \emptyset$. 

\begin{Thm} \label{T:symmetry-uniqueness}
Suppose $P$ is a random walk on an infinite group $G$ induced by a finitely supported measure $\mu$, and assume that $P$ has SRLP. Suppose that $\partial_{\mathrm{R}} G \neq \emptyset$ and that the action of $G$ on $\partial_{\mathrm{R}} G$ is minimal. Then $\J_{\lambda}$ is the largest $\lambda$-invariant proper ideal of $\T(G,\mu)$.
\end{Thm}

\begin{proof}
Let $\J$ be a $\lambda$-invariant proper ideal. We denote by $\overline{\J}$ the image of $\J$ under the quotient map $q_P$. Then, there are two cases.

Suppose first that $\J(G,\mu) \subseteq \J$. Then, we get that $\overline{\J}$ is a proper ideal in $C(\mathrm{R}(G,\mu) \times \mathbb{T}) \otimes \mathbb{K}(\ell^2(G))$. Hence, there exists an open $\lambda$-invariant set $Y \subseteq \mathrm{R}(G,\mu) \times \mathbb{T}$ such that $\overline{\J} = C(Y) \otimes \mathbb{K}(\ell^2(G))$. Then, we must have that $Y \subseteq [G / R_{\mu}] \times \mathbb{T}$. Indeed, if not, then there exists $(\xi,\zeta) \in Y \cap [\partial_{\mathrm{R}} G \times \mathbb{T}]$. Now, since $G$ acting on $\partial_{\mathrm{R}} G$ is minimal, we get that the $G \times \mathbb{T}$ action on $\partial_{\mathrm{R}} G \times \mathbb{T}$ is also minimal. Hence, we get that $Y \supseteq \partial_{\mathrm{R}} G \times \mathbb{T}$. But now, since $Y$ is open, it must contain some element $(hR_{\mu}, \zeta) \in [G / R_{\mu}] \times \mathbb{T}$, and as $\lambda_{g,1}$ acts as left multiplication on $G / R_{\mu}$ together with $\lambda$-invariance of $Y$ we get that $Y = \mathrm{R}(G,\mu)$. Hence, we obtain that $\overline{\J} = C(\mathrm{R}(G,\mu) \times \mathbb{T}) \otimes \mathbb{K}(\ell^2(G))$ in contradiction to $\overline{\J}$ being a proper ideal. Thus, we have shown that $Y \subseteq [G / R_{\mu}] \times \mathbb{T}$, and we obtain that $\J \subseteq \J_{\lambda}$.

Now suppose that $\J$ is a general $\lambda$-invariant proper ideal. Then, since $\overline{\J}$ is proper in $\O(G,\mu)$, we get that $\J + \J(G,\mu)$ is also proper in $\T(G,\mu)$. Hence, by the previous argument we see that $\J \subseteq \J + \J(G,\mu) \subseteq \J_{\lambda}$.
\end{proof}

Recall that a discrete group $G$ is said to be \emph{hyperbolic} if all geodesic triangles in its Cayley graph are $\delta$-thin for some $\delta>0$. This turns out to be independent of the finite set of generators for $G$. The \emph{Gromov boundary} $\partial G$ of $G$ is a compact metrizable $G$-space comprised of equivalence classes of geodesic rays under the equivalence relation of uniform bounded time-distance. A combination of \cite[Proposition 1.13 \& Proposition 3.3]{Bow99} (see also \cite[Remark 5.6]{KK17}) shows that the action of $G$ on $\partial G$ is minimal. For more on the theory of hyperbolic graphs and their boundaries in the context of random walks, we refer to \cite[Section 22 \& Section 27]{Woe00}.

In work of Gou\"{e}zel and Lalley \cite{GL13} and Gou\"{e}zel \cite{Gou14}, it is shown, via a local-limit theorem, that every symmetric aperiodic random walk $P$ on a non-elementary hyperbolic group $G$ satisfies SRLP. From \cite[Corollary 6.6(b)]{Woe+} we get that $R_{\mu}$ is finite, so combined with \cite[Theorem 4.5]{Woe+} we get that the quotient map $G \rightarrow G / R_{\mu}$ induces a homeomorphism $\tau : \partial G \rightarrow \partial_{\mathrm{R}} G$ which is automatically $G$-equivariant. Thus, we obtain the following corollary, showing the existence of a unique smallest $G \times \mathbb{T}$-equivariant quotient for Toeplitz algebras of symmetric random walks on hyperbolic groups.

\begin{Cor} \label{c:sut}
Let $P$ be a symmetric aperiodic random walk on a non-elementary hyperbolic group $G$ induced by a finitely supported $\mu$. Then $C(\partial G \times \mathbb{T}) \otimes \mathbb{K}(\ell^2(G))$ is the unique smallest $G\times \mathbb{T}$ equivariant quotient of $\T(G,\mu)$.
\end{Cor}

Hence, in many examples $\O(G,\mu) \cong C(\mathrm{R}(G,\mu) \times \mathbb{T}) \otimes \mathbb{K}(\ell^2(G))$ fails to be the unique smallest $G \times \mathbb{T}$ equivariant quotient, even when one such exists.

Our final goal is to show that the Toeplitz algebra $\T(G,\mu)$ arises as a Toeplitz algebra of a subproduct system associated to the random walk. We define subproduct systems in the restricted context where the coefficient C*-algebra is $c_0(\mathcal{X})$ for a countable set $\mathcal{X}$ (see \cite[Definition 1.4]{Vis11} for the general definition). We will say that a space $X$ is a correspondence (over $c_0(\mathcal{X})$) if it is a $c_0(\mathcal{X})$-bimodule together with a right-compatible $c_0(\mathcal{X})$-valued inner product so that the left action of $c_0(\mathcal{X})$ is a $*$-homomorphism into bounded operators respecting the right-module structure on $X$. For more on subproduct systems over C*-algebras we recommend \cite{Vis11, Vis12}, and for the theory of C*-correspondences we recommend \cite{Lan95}.

\begin{Def} \label{d:subproductsystem}
Let $\mathcal{X}$ be a countable set. A \emph{subproduct system} is a family $X=\{X_n\}$ of correspondences (over $c_0(\mathcal{X})$) such that
\begin{enumerate}
\item $X_0 = c_0(\mathcal{X})$

\item For all $n,m \in \bN$ there are coisometric bimodule maps 
$$
U_{n,m} : X_n \otimes X_m \rightarrow X_{n+m}
$$
such that $U_{0,n}$ and $U_{n,0}$ are the left and right actions of the bimodule $X_n$, and for all $n,m,\ell \in \bN$ we have the associativity condition
$$
U_{n+m,\ell}(U_{n,m} \otimes I_{X_{\ell}}) = U_{n,m+\ell}(I_{X_n} \otimes U_{m,\ell}).
$$
\end{enumerate}
\end{Def}

Given a subproduct system $X = \{X_n\}$ as above, we may form its Fock space $\F_X := \oplus_{m=0}^{\infty}X_m$ C*-correspondence, as well as the bounded bimodule operators $S^{(n)}_{\xi}$ aon $\F_X$ for $\xi \in X_n$, so that $S^{(n)}_{\xi} : X_m \rightarrow X_{n+m}$ is given by setting $S^{(n)}_{\xi}(\eta) = U_{n,m}(\xi \otimes \eta)$. Denote by $\L(\F_X)$ all bounded right module maps on $\F_X$. The Toeplitz algebra of $X$ is then the C*-subalgebra of $\L(\F_X)$ given by
$$
\T(X) := C^*( \ S^{(n)}_{\xi} \ | \ \xi \in X_n, \ n\in \bN \ ).
$$

Now let $P$ be a random walk on a group $G$ induced by a finitely supported measure $\mu$. We define the correspondences
$$
\Arv_0(P^n) = \{ \ [a_{x,y}] \in c_0(G\times G) \ | \ a_{x,y} = 0 \ \ \text{if} \ (x,y)\notin E(P^n) \  \}.
$$
together with the $c_0(G)$-valued inner product $\langle A, B\rangle = \Diag(A^*B)$, and left and right bimodule actions of $c_0(G)$ given by left and right diagonal matrix multiplication. Note also that each $\Arv_0(P^n)$ is the closed linear span of matrix units $e_{x,y}$ for $(x,y) \in E(P^n)$.

The operation $U_{n,m} : \Arv_0(P^n) \otimes \Arv_0(P^m) \rightarrow \Arv_0(P^{n+m})$ is given on matrix units $e_{x,y} \in \Arv_0(P^n)$ and $e_{y',z} \in \Arv_0(P^m)$ by the rule
$$
U_{n,m}(e_{x,y} \otimes e_{y',z}) = \delta_{y,y'}\sqrt{\frac{P^{(n)}_{x,y}P^{(m)}_{y,z}}{P^{(n+m)}_{x,z}}}e_{x,z}.
$$
It follows from local finiteness of $\Gr(P)$ together with \cite[Theorem 3.4]{DOM14} that $\Arv^P_0:=\{\Arv_0(P^n)\}$ together with $\{U_{n,m}\}$ is a subproduct system. More precisely, $\Arv_0(P^n)$ are correspondences with the above inner product and bimodule actions, and the above rule for $U_{n,m}$ yields a well-defined coisometric bimodule map satisfying the conditions in Definition \ref{d:subproductsystem}.

Now, since $c_0(G)$ is represented as diagonal matrix multiplication on $\ell^2(G)$, by \cite[Corollary 2.74]{RW98} we get a faithful $*$-representation on Hilbert space $\rho: \L(\F_{\Arv^P_0}) \rightarrow B(\F_{\Arv^P_0} \otimes \ell^2(G))$ given by $\rho(T)(\xi\otimes h) = T\xi \otimes h$. We may then identify the the space $\F_{\Arv^P_0} \otimes \ell^2(G)$ with $\F_P$ via the unitary identification $e_{x,z} \otimes e_z \mapsto e^{(m)}_{x,z}$ for $e_{x,z} \in \Arv_0(P^m)$. Under this identification, $\Arv_0(P^m)$ is identified with $\F_P^{(m)}$, and $\F_{\Arv^P_0} \otimes \mathbb{C} e_z$ is identified with $\F_{P,z}$. Hence, we get that $\F_{\Arv^P_0} \otimes \ell^2(G) \cong \F_P$, so that the representation $\rho$ maps $S^{(n)}_{e_{x,y}}$ (which initially acts on $\F_{\Arv^P_0}$) to $S^{(n)}_{x,y}$ acting on $\F_P$. thus, $\rho$ restricts to a $*$-isomorphism from $\T(\Arv^P_0)$ onto $\T(G,\mu)$ (see also \cite[Notation 3.2]{DOM14} and the preceding discussion). The following then coincides with of Viselter's ideal in \cite[Theorem 2.5]{Vis12} by virtue of \cite[Corollary 2.7]{Vis12}.

\begin{Def} \label{d:viselter-cuntz-algebra}
Let $X = \{X_n\}$ be a subproduct system. Viselter's ideal $\I \lhd \T(X)$ is given 
$$
\I_X := \{ \ T \in \T(X) \ | \ \lim_m \|T Q_{[m,\infty)} \| = 0 \},
$$
Where $Q_{[m,\infty)} = \sum_{\ell=m}^{\infty}Q_m$, and $Q_m$ is the natural orthogonal projection from $\F_X$ onto $\Arv_0(P^m)$. Viselter's Cuntz-Pimsner algebra is defined as 
$$
\O(X) = \T(X) / \I_X.
$$
\end{Def}

\begin{Prop} \label{p:viselter-subprod-ideal}
Let $P$ be a random walk on a group $G$ induced by a finitely supported measure $\mu$. Then $\I_{\Arv_0^P} \cong \oplus_{z\in G}\mathbb{K}(\F_{P,k})$.
\end{Prop}

\begin{proof}
Let $\rho : \T(\Arv_0^P) \rightarrow \T(G,\mu)$ be the isomorphism in the discussion preceding Definition \ref{d:viselter-cuntz-algebra}. Then we get that
$$
\rho(\I_X) = \{ \ T \in \T(G,\mu) \ | \ \lim_m \ \| T Q^{[m,\infty)} \| = 0 \ \},
$$
where now $Q^{[m,\infty)}$ is the projection from $\F_P$ onto $\oplus_{\ell=m}^{\infty} \F_P^{(\ell)}$ appearing after Definition \ref{D:RW-cuntz-algebra}. Clearly $\oplus_{z\in G} \mathbb{K}(\F_{P,k}) \subseteq \rho(\I_X)$, and from Proposition \ref{p:viselter-ideal} we get that $Q^{(\ell)}_z \in \oplus_{z\in G} \mathbb{K}(\F_{P,k})$ for every $\ell \in \bN$ and $z\in G$. 

For the converse inclusion, let $T \in \rho(\I_X)$. For a finite set $F\subseteq G$ we let $p_F = \sum_{x\in F} p_x$, and note that $\{p_F\}$ is an approximate identity for $\T(G,\mu)$. Hence, it suffices to show that $Tp_F \in \oplus_{z\in G} \mathbb{K}(\F_{P,k})$ for every finite subset $F\subset G$. But then,
$$
\|Tp_F - Tp_F \cdot [\sum_{\ell=0}^{m-1}Q^{(\ell)}] \| = \|Tp_F Q^{[m,\infty)} \| \rightarrow 0,
$$
and since $Tp_F \cdot [\sum_{\ell=0}^{m-1}Q^{(\ell)}] = Tp_F \cdot \sum_{\ell=0}^{m-1}\sum_{x\in F} Q^{(\ell)}_x \in \oplus_{z\in G} \mathbb{K}(\F_{P,k})$, we get that $Tp_F \in \oplus_{z\in G} \mathbb{K}(\F_{P,k})$.
\end{proof}

In \cite[Section 6, Question 1]{Vis12} Viselter asked whether there is some kind of universality of $\O(X)$ in the spirit of a gauge-invariant uniqueness theorem. By Corollary \ref{c:sut} we get that for symmetric random walks on non-elementary hyperbolic groups the quotient $\T(G,\mu) / \J_{\lambda} \cong C(\partial_{\mathrm{R}} G \times \mathbb{T}) \otimes \mathbb{K}(\ell^2(G))$ satisfies a $G \times \mathbb{T}$-uniqueness theorem even though it is a proper quotient of $\O(G,\mu) \cong C(\mathrm{R}(G,\mu) \times \mathbb{T}) \otimes \mathbb{K}(\ell^2(G))$, and hence of $\O(\Arv_0^P)$. Thus, even with additional natural symmetries that enable the existence of a unique smallest symmetry-equivariant quotient, the quotient by $\J_{\lambda}$ fails to coincide with Viselter's Cuntz-Pimsner algebra of the subproduct system.

When $\mathrm{R}(G,\mu) = G / R_{\mu}$, then $\partial_{\mathrm{R}}G = \emptyset$ and the action of $G$ on $\mathrm{R}(G,\mu)$ is minimal. In this case, we can deduce similarly that $\O(G,\mu)$ is the unique smallest $G \times \mathbb{T}$-equivariant quotient of $\T(G,\mu)$. Theorem \ref{T:symmetry-uniqueness} then motivates the following question in the complementary case

\begin{Ques}
Let $P$ be a random walk on $G$ induced by a measure $\mu$. Suppose $P$ has SRLP and that $\partial_{\mathrm{R}} G \neq \emptyset$. Is there a unique smallest $G \times \mathbb{T}$-equivariant quotient of $\T(G,\mu)$? Better yet, is the action of $G$ on $\partial_{\mathrm{R}} G$ always minimal?
\end{Ques}

%%%%%%%%%%%%%%%%%%%%%%%%%%%%%%%%%%%%%%%%%%%

\end{document}